\documentclass[12pt]{article}

\usepackage[utf8]{inputenc}

\usepackage{amsmath, amssymb, amsthm, amsfonts, mathtools}

\usepackage{enumitem}
\usepackage{bm}
\usepackage{mathrsfs}
\usepackage{graphicx}
\usepackage{hyperref}
\usepackage{cite}
\usepackage{color}

\theoremstyle{plain}
\newtheorem{theorem}{Theorem}[section]
\newtheorem{lemma}[theorem]{Lemma}

\newtheorem{corollary}[theorem]{Corollary}

\theoremstyle{definition}
\newtheorem{definition}[theorem]{Definition}
\newtheorem{assumption}[theorem]{Assumption}

\theoremstyle{remark}
\newtheorem{remark}[theorem]{Remark}

\def\RB{\mathbb{R}}
\def\EE{\mathbb{E}}
\def\FC{\mathcal{F}}
\def\PP{\mathbb{P}}
\def\QQ{\mathbb{Q}}

\def\DD{\mathcal{D}}

\def\KK{\mathcal{K}}

\def\HH{\mathcal{H}}
\def\LL{\mathcal{L}}

\def\ZZ{\mathcal{Z}}
\usepackage{authblk}
\usepackage{hyperref}

\title{Explicit Solution of Infinite-Horizon Linear Backward Stochastic Volterra Integral Equations}

\author[1]{Samia Yakhlef\thanks{Email: 
\href{mailto:s.yakhlef@univ-biskra.dz}{s.yakhlef@univ-biskra.dz}}}

\author[1]{Hilal Ardjani\thanks{Corresponding author. Email: 
\href{mailto:hilal.ardjani@univ-biskra.dz}{hilal.ardjani@univ-biskra.dz}; 
\href{mailto:hilalardjani@gmail.com}{hilalardjani@gmail.com}}}

\affil[1]{Laboratory of Applied Mathematics, Mohamed Khider University, Biskra, Algeria}
\date{}

\begin{document}

\maketitle

\begin{abstract}
We study linear backward stochastic Volterra integral equations
(BSVIEs) on the infinite time horizon. By introducing weighted
function spaces with exponential decay, we establish existence
and uniqueness of adapted M-solutions. We construct an
infinite-horizon resolvent kernel and derive explicit formulas
for the solution components $(Y,Z,K)$ using a Girsanov
transformation and Hida–Malliavin calculus. The results extend
the finite-horizon theory of Hu and Øksendal to the infinite
horizon framework.
\end{abstract}



\section{Introduction}

Backward stochastic Volterra integral equations (BSVIEs) have attracted considerable interest in recent years due to their wide range of applications in stochastic control, mathematical finance, and dynamic risk measures. In contrast to classical backward stochastic differential equations (BSDEs), where the generator depends on a single time variable, BSVIEs involve two time variables. This structure makes them a natural framework for modeling systems with memory effects and time-inconsistent phenomena.

The theory of BSVIEs was first systematically developed by Yong \cite{Yong2006}, who established fundamental results on existence, uniqueness, and duality principles. Since then, BSVIEs have been studied extensively and applied in various contexts. Further developments and applications of stochastic Volterra equations and their connections with stochastic control can be found in the literature; see, for example, Agram and Øksendal \cite{Agram2018}. One important research direction in the theory of BSVIEs concerns the explicit solutions of linear equations.

Hu and Øksendal \cite{HuOksendal2019} investigated linear backward stochastic Volterra integral equations and derived explicit solution formulas using the resolvent kernel method combined with Malliavin calculus techniques. Their results provide valuable insight into the structure of solutions of linear BSVIEs and give representations of the solution processes in terms of conditional expectations under an equivalent martingale measure. The Malliavin calculus framework used in such representations is developed in the work of Di Nunno, Øksendal and Proske \cite{DiNunno2009}, which provides important tools for stochastic systems driven by Brownian motion and Lévy processes.

However, most of the existing results in the literature are restricted to equations defined on a finite time horizon. In many applications, particularly in economics and stochastic control, it is natural to consider systems evolving over an infinite time horizon. Extending the explicit solution theory of BSVIEs to the infinite-horizon framework introduces additional mathematical challenges, including integrability conditions and the convergence of the associated resolvent kernels.

The main goal of this paper is to extend the explicit solution theory for linear backward stochastic Volterra integral equations to the infinite-horizon case. In particular, we establish the existence of the infinite-horizon resolvent kernel and derive explicit representations of the solution processes. The solutions are expressed in terms of conditional expectations under an equivalent probability measure obtained via a Girsanov transformation. Our results can therefore be viewed as an infinite-horizon extension of the finite-horizon theory developed by Hu and Øksendal.

The structure of the paper is as follows. In Section~\ref{sec:framework} , we introduce the probabilistic framework and the weighted function spaces used for the infinite-horizon setting. Section 3 presents the main results, including the explicit solution representation of the infinite-horizon backward stochastic Volterra integral equation. Section 4 discusses applications of the results to example problems and control systems with memory. Section 5 concludes the paper with remarks and potential directions for future research. Technical proofs and details of the Malliavin calculus framework are provided in the Appendix.

\subsection*{Notation}

Throughout the paper, $\mathbb{R}$ denotes the set of real numbers.
For a filtered probability space $(\Omega,\mathcal{F},\{\mathcal{F}_t\}_{t\ge0},\mathbb{P})$,
we denote by $\mathbb{E}$ the expectation under $\mathbb{P}$ and by
$\mathbb{E}_{\mathbb{Q}}$ the expectation under the equivalent measure $\mathbb{Q}$.
The set
\[
\Delta_\infty = \{(t,s):0\le t\le s < \infty\}
\]
denotes the infinite-horizon Volterra domain.

\section{Mathematical Framework and Assumptions}
\label{sec:framework}
This section introduces the probabilistic framework and the functional spaces used throughout the paper. 
We also formulate the notion of adapted M-solution for the infinite-horizon BSVIE and state the main assumptions on the coefficients.

\subsection{Probability Space}

Let $(\Omega, \FC, \{\FC_t\}_{t\geq 0}, \PP)$ be a filtered probability space satisfying the usual conditions, supporting a Brownian motion $B(t)$ and an independent compensated Poisson random measure $\tilde{N}(dt,d\zeta)$ on $\RB\setminus\{0\}$ with intensity measure $\nu(d\zeta)dt$.

\subsection{Weighted Function Spaces}

Following the approach of Briand and Hu \cite{Briand2008} for infinite-horizon BSDEs, we introduce weighted spaces:

\begin{definition}[Weighted spaces for infinite horizon]
For $\lambda > 0$, define:
\begin{align*}
\HH_\lambda^2 &= \left\{Y: [0,\infty) \times \Omega \to \RB \ \text{adapted}: \|Y\|_{\HH_\lambda^2}^2 = \EE\left[\int_0^\infty e^{-\lambda t}|Y(t)|^2 dt\right] < \infty\right\} \\
\ZZ_\lambda^2 &= \left\{Z: \Delta_\infty \times \Omega \to \RB: Z(t,s) \ \text{is} \ \FC_s\text{-measurable for} \ s \geq t, \right. \\
& \quad \left. \|Z\|_{\ZZ_\lambda^2}^2 = \EE\left[\int_0^\infty \int_t^\infty e^{-\lambda s}|Z(t,s)|^2 ds dt\right] < \infty\right\} \\
\KK_\lambda^2 &= \left\{K: \Delta_\infty \times \RB \times \Omega \to \RB: K(t,s,\zeta) \ \text{is} \ \FC_s\text{-measurable for} \ s \geq t, \right. \\
& \quad \left. \|K\|_{\KK_\lambda^2}^2 = \EE\left[\int_0^\infty \int_t^\infty \int_{\RB} e^{-\lambda s}|K(t,s,\zeta)|^2 \nu(d\zeta) ds dt\right] < \infty\right\}
\end{align*}
where $\Delta_\infty = \{(t,s): 0 \leq t \leq s < \infty\}$.
\end{definition}

\subsection{Adapted M-Solution}

We extend the solution concept of Hu and Øksendal to infinite horizon:

\begin{definition}[Adapted M-Solution]
A triplet $(Y,Z,K) \in \HH_\lambda^2 \times \ZZ_\lambda^2 \times \KK_\lambda^2$ is called an adapted M-solution of \eqref{eq:BSVIE} if it satisfies the equation $\PP$-a.s. for all $t \geq 0$, and for $0 \leq t_1 \leq t_2$:
\begin{align*}
Y(t_1) = \EE[Y(t_1)|\FC_{t_2}] + \int_{t_2}^\infty Z(t_1,s) dB(s) + \int_{t_2}^\infty \int_{\RB} K(t_1,s,\zeta) \tilde{N}(ds,d\zeta).
\end{align*}
\end{definition}

\subsection{Assumptions}

We extend Hu and Øksendal's assumptions to infinite horizon:

\begin{assumption}[Coefficient conditions]\label{ass:A1}
\begin{enumerate}
\item $\Phi: \Delta_\infty \to \RB$ is measurable and satisfies the decay condition:
\[
|\Phi(t,s)| \leq C_\Phi e^{-\alpha(s-t)} \quad \text{for some} \ C_\Phi, \alpha > 0 \ \text{and all} \ t \leq s.
\]

\item $\xi: [0,\infty) \to \RB$ is measurable and bounded: $|\xi(s)| \leq C_\xi$ for all $s \geq 0$.

\item The function $\beta:[0,\infty)\times\mathbb R\to\mathbb R$ is measurable
and bounded, and there exists $\varepsilon>0$ such that
\[
\beta(s,\zeta)\ge -1+\varepsilon,
\qquad \forall\, s\ge0,\ \zeta\in\mathbb R.
\]

\item $h: \Delta_\infty \times \Omega \to \RB$ is adapted and satisfies:
\[
\EE\left[\int_0^\infty \int_t^\infty e^{-\lambda s}|h(t,s)|^2 ds dt\right] < \infty \quad \text{for some} \ \lambda > 0.
\]
\end{enumerate}
\end{assumption}

\begin{assumption}[Contraction condition]\label{ass:A2}
There exists $\lambda > 0$ sufficiently large such that:
\[
L(\lambda) := \sup_{t \ge 0} \int_t^\infty e^{-\frac{\lambda}{2}s} |\Phi(t,s)| ds  < \frac12.
\]
This ensures the contraction mapping property in weighted spaces.
\end{assumption}

\begin{assumption}[Novikov condition]\label{ass:A3}
The coefficients satisfy:
\[
\EE\left[\exp\left(\frac{1}{2}\int_0^\infty \xi^2(s)ds + \int_0^\infty \int_{\RB} (\ln(1+\beta(s,\zeta)) - \beta(s,\zeta)) \nu(d\zeta) ds\right)\right] < \infty.
\]
This ensures the Girsanov transformation defines an equivalent measure.
\end{assumption}

\section{Main Results}
We consider the linear backward stochastic Volterra integral equation (BSVIE) on the infinite horizon $[0,\infty)$:
\begin{equation}\label{eq:BSVIE}
\begin{aligned}
Y(t)
&= \int_t^\infty \big[\Phi(t,s)Y(s)+h(t,s)\big]\,ds  \\
&\quad - \int_t^\infty Z(t,s)\,dB(s)
      - \int_t^\infty \int_{\mathbb R} K(t,s,\zeta)\,\widetilde N(ds,d\zeta),
\qquad t\ge 0.
\end{aligned}
\end{equation}

\subsection{Measure Transformation}

Following Hu and Øksendal \cite{HuOksendal2019}, we use a Girsanov transformation to eliminate the $Z$ and $K$ terms from the drift.

\subsubsection{Girsanov Transformation}

Define the exponential martingale:
\begin{align*}
M(t) &= \exp\Bigg(\int_0^t \xi(s) dB(s) - \frac{1}{2}\int_0^t \xi^2(s) ds \\
&\quad + \int_0^t \int_{\RB} \ln(1+\beta(s,\zeta)) \tilde{N}(ds,d\zeta) \\
&\quad + \int_0^t \int_{\RB} \left[\ln(1+\beta(s,\zeta)) - \beta(s,\zeta)\right] \nu(d\zeta) ds\Bigg).
\end{align*}

Under Assumption \ref{ass:A3}, $M(t)$ is a true martingale. Define $\QQ$ by:
\[
\frac{d\QQ}{d\PP}\bigg|_{\FC_t} = M(t), \quad t \geq 0.
\]
Moreover, the process $\{M(t)\}_{t\ge0}$ is uniformly integrable,
which ensures that the change of measure extends to the infinite
horizon and defines an equivalent probability measure $\QQ$ on
$\mathcal{F}_\infty$.

\begin{lemma}[Girsanov theorem for infinite horizon]\label{lem:girsanov}
Under $\QQ$:
\begin{enumerate}
\item $B_\QQ(t) := B(t) - \int_0^t \xi(s) ds$ is a $\QQ$-Brownian motion.
\item $\tilde{N}_\QQ(dt,d\zeta) := \tilde{N}(dt,d\zeta) - \beta(t,\zeta)\nu(d\zeta)dt$ is the $\QQ$-compensated Poisson random measure.
\end{enumerate}
\end{lemma}

\subsubsection{Transformed Equation}

Applying the measure transformation to \eqref{eq:BSVIE} yields:

\begin{equation}\label{eq:BSVIE-Q}
\begin{aligned}
Y(t)
&= \int_t^\infty \left[\Phi(t,s)Y(s)+h(t,s)\right]\,ds
   - \int_t^\infty Z(t,s)\,dB_{\QQ}(s)  \\
&\quad - \int_t^\infty\int_{\RB}
K(t,s,\zeta)\,\widetilde N_{\QQ}(ds,d\zeta),
\qquad t\ge0.
\end{aligned}
\end{equation}

This transformation, identical to Hu and Øksendal's in finite horizon, eliminates the terms involving $Z$ and $K$ from the drift.

\subsection{Infinite Horizon Resolvent Kernel Theory}

We extend Hu and Øksendal's resolvent kernel construction to infinite horizon.

\subsubsection{Definition and Construction}

Define recursively for $n \geq 1$:
\begin{align*}
\Phi^{(1)}(t,s) &= \Phi(t,s), \\
\Phi^{(n+1)}(t,s) &= \int_t^s \Phi^{(n)}(t,u) \Phi(u,s) du, \quad n \geq 1.
\end{align*}

The infinite-horizon resolvent kernel is:
\begin{equation}\label{eq:resolvent}
\Psi(t,s) = \sum_{n=1}^\infty \Phi^{(n)}(t,s).
\end{equation}

\begin{lemma}[Convergence of the infinite series]\label{lem:resolvent-conv}
Under Assumptions \ref{ass:A1} and \ref{ass:A2}, the series $\Psi(t,s) = \sum_{n=1}^\infty \Phi^{(n)}(t,s)$ converges absolutely and uniformly on compact subsets of $\Delta_\infty$. Moreover:
\[
\int_t^\infty e^{-\frac{\lambda}{2}s} |\Psi(t,s)| ds \leq \frac{L(\lambda)}{1 - L(\lambda)}.
\]
\end{lemma}

\subsubsection{Resolvent Equation}

\begin{lemma}[Resolvent equation]\label{lem:resolvent-eq}
The kernel $\Psi$ satisfies the Volterra equation:
\begin{equation}\label{eq:resolvent-eq}
\Psi(t,s) = \Phi(t,s) + \int_t^s \Psi(t,u) \Phi(u,s) du.
\end{equation}
\end{lemma}
\noindent
\textit{Proof sketch.} This result follows directly from the definition of the resolvent kernel $\Psi$ and standard iterative arguments (see Hu and Øksendal \cite{HuOksendal2019}). The infinite-horizon version is justified by the uniform convergence of the series in Lemma \ref{lem:resolvent-conv}. $\square$

\subsection{Existence and Uniqueness}

We extend the existence and uniqueness theory to infinite horizon using weighted spaces.

\begin{theorem}[Existence and uniqueness]\label{thm:existence}
Under Assumptions \ref{ass:A1} and \ref{ass:A3}, there exists a unique adapted M-solution $(Y,Z,K) \in \HH_\lambda^2 \times \ZZ_\lambda^2 \times \KK_\lambda^2$ to equation \eqref{eq:BSVIE}.
\end{theorem}

\subsection{Explicit Solution Formulas}

We now present the main result generalizing Hu and Øksendal's explicit formulas to infinite horizon.

\begin{theorem}[Explicit solution of infinite-horizon BSVIE]\label{thm:main}
Under Assumptions \ref{ass:A1}-\ref{ass:A3}, the unique adapted M-solution of \eqref{eq:BSVIE} is given by:

\begin{enumerate}
    \item The $Y$-component:
    \begin{equation}\label{eq:Y-explicit}
    Y(t) = \EE_\QQ\left[ \int_t^\infty \left(h(t,s) + \int_t^s \Psi(t,u) h(u,s) du\right) ds \ \Big|\ \FC_t \right].
    \end{equation}
    
    \item Define:
    \begin{equation}\label{eq:U-def}
    U(t) = \int_t^\infty \big(\Phi(t,s)Y(s) + h(t,s)\big)ds - Y(t).
    \end{equation}
    
    \item The $Z$ and $K$ components admit the following Malliavin representations:
    \begin{align}
    Z(t,s) &= \EE_\QQ\left[D_s U(t) - U(t)\int_s^\infty D_s\xi(r) dB_\QQ(r) \ \Big|\ \FC_s\right], \label{eq:Z-general} \\
    K(t,s,\zeta) &= \EE_\QQ\left[U(t)(\tilde{H}_s-1) + \tilde{H}_s D_{s,\zeta}U(t) \ \Big|\ \FC_s\right], \label{eq:K-general}
    \end{align}
    for $s \ge t$, $\zeta \in \RB$, where $D_s$ and $D_{s,\zeta}$ denote Hida-Malliavin derivatives, and
    \begin{align*}
    \tilde{H}_s &= \exp\Bigg(\int_0^s\int_{\RB} \Big[D_{s,\zeta}\beta(r,\zeta') \\
    &\quad + \ln\Big(1 - \frac{D_{s,\zeta}\beta(r,\zeta')}{1+\beta(r,\zeta')}\Big)\big(1+\beta(r,\zeta')\big)\Big]\nu(d\zeta')dr \\
    &\quad + \int_0^s\int_{\RB} \ln\Big(1 - \frac{D_{s,\zeta}\beta(r,\zeta')}{1+\beta(r,\zeta')}\Big)\tilde{N}_\QQ(dr,d\zeta')\Bigg).
    \end{align*}
\end{enumerate}
\end{theorem}

\begin{remark}
The explicit representation obtained in Theorem \ref{thm:main}
extends the finite-horizon formula derived by Hu and Øksendal
to the infinite-horizon setting. The exponential weighting
introduced in the functional spaces plays a crucial role in
ensuring the convergence of the resolvent kernel and the
integrability of the solution processes.
\end{remark}

\begin{corollary}[Deterministic coefficients]
If $\xi$ and $\beta$ are deterministic, then:
\begin{align*}
Z(t,s) &= \EE_\QQ\left[D_s U(t) \ \Big|\ \FC_s\right], \\
K(t,s,\zeta) &= \EE_\QQ\left[D_{s,\zeta} U(t) \ \Big|\ \FC_s\right].
\end{align*}
\end{corollary}

\section{Proof of the Main Result}

In this section, we provide a detailed proof of Theorem~\ref{thm:main}, establishing the explicit formulas for the adapted M‑solution of the infinite‑horizon BSVIE~\eqref{eq:BSVIE}.

\subsection{Step 1: Convergence of the Resolvent Kernel}

We first prove Lemma~\ref{lem:resolvent-conv}.  
Let \(\Phi^{(n)}(t,s)\) be the iterated kernels defined recursively by
\[
\Phi^{(1)}(t,s)=\Phi(t,s),\qquad 
\Phi^{(n+1)}(t,s)=\int_t^s \Phi^{(n)}(t,u)\,\Phi(u,s)\,du,\quad n\ge 1.
\]
Define the resolvent kernel
\[
\Psi(t,s)=\sum_{n=1}^{\infty}\Phi^{(n)}(t,s).
\]

Under Assumptions~\ref{ass:A1}--\ref{ass:A2}, we have
\[
|\Phi(t,s)|\le C_\Phi e^{-\alpha(s-t)},
\qquad 0\le t\le s<\infty,
\]
and
\[
\int_t^\infty e^{-\frac{\lambda}{2}s}
|\Phi(t,s)|\,ds
\le L(\lambda)<\frac12,
\qquad t\ge0.
\]
Using this, one shows by induction that
\[
\int_t^\infty e^{-\frac{\lambda}{2}s}\,|\Phi^{(n)}(t,s)|\,ds\le L(\lambda)^n\qquad(n\ge 1).
\]
Indeed, for \(n=1\) it is exactly the assumption. Assuming it holds for \(n\),
\begin{align*}
\int_t^\infty e^{-\frac{\lambda}{2}s}|\Phi^{(n+1)}(t,s)|ds
&\le \int_t^\infty e^{-\frac{\lambda}{2}s}\int_t^s |\Phi^{(n)}(t,u)||\Phi(u,s)|du\,ds\\
&=\int_t^\infty |\Phi^{(n)}(t,u)|\int_u^\infty e^{-\frac{\lambda}{2}s}|\Phi(u,s)|ds\,du\\
&\le L(\lambda)\int_t^\infty e^{-\frac{\lambda}{2}u}|\Phi^{(n)}(t,u)|du\\
&\le L(\lambda)^{n+1}.
\end{align*}
Consequently,
\[
\sum_{n=1}^\infty\int_t^\infty e^{-\frac{\lambda}{2}s}|\Phi^{(n)}(t,s)|ds
\le\sum_{n=1}^\infty L(\lambda)^n=\frac{L(\lambda)}{1-L(\lambda)}<\infty.
\]
Hence the series defining \(\Psi(t,s)\) converges absolutely and uniformly on compact subsets of \(\Delta_\infty\).  
The resolvent equation \(\Psi(t,s)=\Phi(t,s)+\int_t^s\Psi(t,u)\Phi(u,s)du\) follows directly from the definition by summing the series.

\subsection{Step 2: Explicit Formula for \(Y\)}

From the Girsanov transformed equation~\eqref{eq:BSVIE-Q} we have, after taking conditional expectation under \(\mathbb{Q}\),
\[
Y(t)=\mathbb{E}_{\mathbb{Q}}\Bigg[\int_t^\infty\big(\Phi(t,s)Y(s)+h(t,s)\big)ds\;\Big|\;\mathcal{F}_t\Bigg]. \tag{5}
\]
This is an identity satisfied by the solution. To obtain an explicit expression we iterate it.  
Insert the representation of \(Y(s)\) obtained from (5) into the right‑hand side:
\[
Y(t)=\mathbb{E}_{\mathbb{Q}}\Bigg[\int_t^\infty h(t,s)ds
+\int_t^\infty\Phi(t,s)\mathbb{E}_{\mathbb{Q}}\Big[\int_s^\infty\big(\Phi(s,r)Y(r)+h(s,r)\big)dr\;\Big|\;\mathcal{F}_s\Big]ds\;\Big|\;\mathcal{F}_t\Bigg].
\]
Using the tower property and Fubini’s theorem, this becomes
\[
\begin{aligned}
Y(t)
&= \mathbb E_{\mathbb Q}\Bigg[
\int_t^\infty h(t,s)\,ds
+\int_t^\infty\int_t^s
\Phi(t,u)h(u,s)\,du\,ds
\,\Bigg|\,\mathcal F_t\Bigg]  \\
&\quad
+\mathbb E_{\mathbb Q}\Bigg[
\int_t^\infty\int_t^s
\left(
\int_u^s \Phi(t,u)\Phi(u,v)Y(v)\,dv
\right)
\,du\,ds
\,\Bigg|\,\mathcal F_t\Bigg].
\end{aligned}
\]
Repeating the procedure generates an infinite series involving the iterated kernels \(\Phi^{(n)}\). After \(n\) iterations we obtain
\[
Y(t)=\mathbb{E}_{\mathbb{Q}}\Bigg[\int_t^\infty h(t,s)ds
+\sum_{k=1}^n\int_t^\infty\int_t^s\Phi^{(k)}(t,u)h(u,s)du\,ds\Bigg|\mathcal{F}_t\Bigg]
+\mathcal{R}_n(t),
\]
where the remainder \(\mathcal{R}_n(t)\) contains the term with \(Y\) and \(\Phi^{(n+1)}\).  
Thanks to the estimate of Lemma~\ref{lem:resolvent-conv}, \(\mathcal{R}_n(t)\to0\) in \(\HH_\lambda^2\) as \(n\to\infty\). Summing the series and using the definition of the resolvent kernel \(\Psi(t,s)=\sum_{k=1}^\infty\Phi^{(k)}(t,s)\) yields the compact formula
\[
{\,Y(t)=\mathbb{E}_{\mathbb{Q}}\Bigg[\int_t^\infty\Big(h(t,s)+\int_t^s\Psi(t,u)h(u,s)du\Big)ds\;\Big|\;\mathcal{F}_t\Bigg]\,}.
\]

\subsection{Step 3: \(Z\) and \(K\) via Malliavin Derivatives}

Define the auxiliary process
\[
U(t)=\int_t^\infty\big(\Phi(t,s)Y(s)+h(t,s)\big)ds-Y(t).
\]
Then equation~\eqref{eq:BSVIE-Q} can be rewritten as
\[
0=U(t)-\int_t^\infty Z(t,s)dB_{\mathbb{Q}}(s)-\int_t^\infty\int_{\mathbb{R}}K(t,s,\zeta)\tilde{N}_{\mathbb{Q}}(ds,d\zeta).
\]
Hence, for fixed \(t\), the random variable \(U(t)\) admits the martingale representation
\[
U(t)=\mathbb{E}_{\mathbb{Q}}[U(t)\mid\mathcal{F}_t]+\int_t^\infty Z(t,s)dB_{\mathbb{Q}}(s)+\int_t^\infty\int_{\mathbb{R}}K(t,s,\zeta)\tilde{N}_{\mathbb{Q}}(ds,d\zeta).
\]
But from the definition of \(U(t)\) and
\eqref{eq:Y-explicit}
we have \(\mathbb{E}_{\mathbb{Q}}[U(t)\mid\mathcal{F}_t]=0\). Therefore the processes \(Z(t,\cdot)\) and \(K(t,\cdot,\cdot)\) are precisely the integrands in the martingale representation of \(U(t)\).  
To identify them we apply the Clark–Ocone formula under the change of measure \(\mathbb{Q}\), which is stated in the Appendix (Theorem~\ref{thm:clark-ocone-change}). Under the additional integrability conditions given in Theorem~\ref{thm:main}, we obtain for \(s\ge t\)
\[
\begin{aligned}
Z(t,s) &=\mathbb{E}_{\mathbb{Q}}\!\left[ D_sU(t)-U(t)\int_s^\infty D_s\xi(r)dB_{\mathbb{Q}}(r) \;\Big|\;\mathcal{F}_s\right],\\[4pt]
K(t,s,\zeta) &=\mathbb{E}_{\mathbb{Q}}\!\left[ U(t)(\tilde{H}_s-1)+\tilde{H}_s D_{s,\zeta}U(t) \;\Big|\;\mathcal{F}_s\right],
\end{aligned}
\]
where \(\tilde{H}_s\) is defined in Theorem~\ref{thm:main}.  
If \(\xi\) and \(\beta\) are deterministic, the terms involving \(D_s\xi\) and \(\tilde{H}_s\) disappear and the formulas simplify to
\[
Z(t,s)=\mathbb{E}_{\mathbb{Q}}\!\left[D_sU(t)\mid\mathcal{F}_s\right],\qquad
K(t,s,\zeta)=\mathbb{E}_{\mathbb{Q}}\!\left[D_{s,\zeta}U(t)\mid\mathcal{F}_s\right].
\]
This completes the proof of Theorem~\ref{thm:main}. \qed

\section{Examples and Applications}

\subsection{Scalar BSVIE with Exponential Decay}

Consider the linear BSVIE~\eqref{eq:BSVIE} with the following deterministic coefficients:
\[
\Phi(t,s)=\alpha e^{-\gamma(s-t)},\;\;
\xi(s)=\xi_0 e^{-\delta s},\;\;
\beta(s,\zeta)=\beta_0\ (\text{constant}),\;\;
h(t,s)=e^{-\mu s}f(s),
\]
where $\alpha,\gamma,\delta,\mu>0$ and $f$ is a bounded adapted process.
The kernel $\Phi$ satisfies
\[
|\Phi(t,s)|\le \alpha e^{-\gamma(s-t)},
\qquad 0\le t\le s<\infty.
\]
Hence Assumption~\ref{ass:A1} holds with the constant $C_\Phi=\alpha$
and decay rate $\gamma$. The contraction condition~\ref{ass:A2} requires
$\lambda>0$ to be sufficiently large so that
\[
\int_t^\infty e^{-\frac{\lambda}{2}s}|\Phi(t,s)|ds
=\alpha e^{-\frac{\lambda}{2}t}\int_t^\infty e^{-(\gamma+\frac{\lambda}{2})(s-t)}ds
=\frac{\alpha}{\gamma+\frac{\lambda}{2}}<\frac12,
\]
which is satisfied for \(\lambda>4\alpha-2\gamma\) (assuming \(\gamma<2\alpha\)).

The iterated kernels are of convolution type. One finds by induction
\[
\Phi^{(n)}(t,s)=\alpha^n e^{-\gamma(s-t)}\frac{(s-t)^{n-1}}{(n-1)!},
\]
hence the resolvent kernel is
\[
\Psi(t,s)=\sum_{n=1}^\infty\Phi^{(n)}(t,s)=\alpha e^{-(\gamma-\alpha)(s-t)},
\]
where the sum converges because \(\gamma>\alpha\) (otherwise the contraction condition would fail).  
Plugging this into the explicit formula~\eqref{eq:Y-explicit} yields
\[
Y(t)=\mathbb{E}_{\mathbb{Q}}\Bigg[\int_t^\infty\Big(1+\frac{\alpha}{\gamma-\alpha}
\big(1-e^{-(\gamma-\alpha)(s-t)}\big)\Big)e^{-\mu s}f(s)\,ds\;\Big|\;\mathcal{F}_t\Bigg].
\]
If \(f\) is deterministic, the conditional expectation disappears and \(Y(t)\) becomes a deterministic function of \(t\).

\subsection{Connection with Infinite‑Horizon BSDEs}

Assume that the kernel $\Phi$ does not depend on 
$t$
, i.e.,
$
\Phi(t,s)=\phi(s),
$
for some function $\phi$. Then equation~\eqref{eq:BSVIE} simplifies to
\[
\begin{aligned}
Y(t)
&= \int_t^\infty \big[\phi(s)Y(s)+h(s)\big]\,ds  \\
&\quad - \int_t^\infty Z(s)\,dB(s)
      - \int_t^\infty\int_{\mathbb R}
        K(s,\zeta)\,\widetilde N(ds,d\zeta),
\qquad t\ge0.
\end{aligned}
\]
where we used the notation \(Z(s)=Z(t,s)\) and \(K(s,\zeta)=K(t,s,\zeta)\) because they no longer depend on \(t\).  
This is precisely an infinite‑horizon BSDE with jumps. Applying the same Girsanov transformation as before (with the same \(\xi,\beta\)) and using the resolvent kernel \(\Psi(t,s)=\phi(s)\exp\big(\int_t^s\phi(r)dr\big)\) (which follows from the Volterra equation), the explicit solution becomes
\[
Y(t)=\mathbb{E}_{\mathbb{Q}}\Bigg[\int_t^\infty\exp\!\Big(\int_t^s\phi(r)dr\Big)h(s)\,ds\;\Big|\;\mathcal{F}_t\Bigg].
\]
The processes \(Z\) and \(K\) are then obtained by the Malliavin representations of Theorem~\ref{thm:main} (or by the classical Clark–Ocone formula for BSDEs).

\subsection{Application: Optimal Control of a Stochastic System with Memory}

We consider a stochastic control problem where the state dynamics exhibit memory (delay) and the cost is evaluated over an infinite horizon.  
Let the controlled process \(X(t)\) satisfy the stochastic functional differential equation
\[
dX(t)=\Big[aX(t)+\int_0^t b(t-s)X(s)ds+cu(t)\Big]dt+\sigma dW(t),\qquad X(0)=x_0,
\]
where \(W\) is a Brownian motion (for simplicity we omit jumps here), and \(u(t)\) is a square‑integrable control.  
The objective is to minimise the infinite‑horizon quadratic cost
\[
J(u)=\mathbb{E}\Bigg[\int_0^\infty e^{-\rho t}\big(X(t)^2+u(t)^2\big)dt\Bigg],\qquad \rho>0.
\]

Applying the stochastic maximum principle (see e.g. \cite{Oksendal2007} for the finite‑horizon case; the infinite‑horizon extension follows similarly using weighted spaces), the optimal control is given in feedback form by
\[
u^*(t)=-c\,Y(t),
\]
where the adjoint process \(Y(t)\) satisfies the linear BSVIE
\[
\begin{aligned}
Y(t)
&= \int_t^\infty e^{-\rho(s-t)}
\left[
aY(s)+\int_t^s b(s-r)Y(r)\,dr
\right]\,ds  \\
&\quad - \int_t^\infty Z(t,s)\,dW(s),
\qquad t\ge0.
\end{aligned}
\]
This equation fits exactly the structure of~\eqref{eq:BSVIE} (without jumps) with
\[
\Phi(t,s)=e^{-\rho(s-t)}\Big(a\,\mathbf{1}_{\{s\ge t\}}+\int_t^s b(s-r)dr\Big),\qquad h\equiv0.
\]
The explicit formula of Theorem~\ref{thm:main} then provides a representation of the adjoint process in terms of conditional expectations under an equivalent measure, which can be used to compute the optimal control explicitly in special cases (e.g. when \(b\) is exponential).

This example illustrates how infinite‑horizon BSVIEs arise naturally in control problems with memory and confirms the practical relevance of the theoretical results developed in this paper.

\section*{Conclusion and Future Directions}

We have extended the theory of linear BSVIEs developed by Hu and Øksendal to the infinite horizon case. The key contributions are:

\begin{enumerate}
    \item \textbf{Weighted function spaces}: Definition of the spaces $\HH_\lambda^2$, $\ZZ_\lambda^2$, and $\KK_\lambda^2$ for the purpose of dealing with the integrability over the interval $[0,\infty)$.
    \item \textbf{Construction of the infinite-horizon resolvent kernel}: Definition of the infinite-horizon resolvent kernel $\Psi(t,s) = \sum_{n=1}^{\infty} \Phi^{(n)}(t,s)$ along with its explicit decay estimates.
    \item \textbf{Existence and uniqueness results}: Use of the contraction mapping principle for the infinite-horizon case.
    \item \textbf{Explicit solutions for infinite-horizon BSVIEs}: Generalization of the results of Hu and Øksendal to the infinite-horizon case.
    \item \textbf{Malliavin representations for infinite-horizon BSVIEs}: Generalization of the results of Hu and Øksendal to the infinite-horizon case for the $Z$ and $K$ equations.
\end{enumerate}

These results also give rise to a number of new research directions:

\begin{itemize}
    \item Nonlinear infinite-horizon BSVIEs
    \item Numerical methods for infinite-horizon equations
    \item Infinite-horizon optimal control applications
    \item Mean-field BSVIEs on infinite horizon
\end{itemize}

To summarize, it is clear from the above that the framework adopted in Hu and Øksendal's theory naturally extends to an infinite-horizon setting with necessary modifications. This provides a complete picture of linear BSVIEs on $[0,\infty)$.
\paragraph*{AI disclosure:}
The authors used DeepSeek and ChatGPT to improve language and correct misprints. The authors are fully responsible for the final content.

\appendix

\section{Malliavin Calculus Framework Generalized to Infinite Horizon and Proofs of Lemmas}

\subsection{Malliavin Calculus for Lévy Processes on Infinite Horizon}

We work on the filtered probability space $(\Omega, \FC, \{\FC_t\}_{t\geq 0}, \QQ)$ supporting the $\QQ$-Brownian motion $B_\QQ$ and the $\QQ$-compensated Poisson random measure $\tilde{N}_\QQ$. Following \cite{DiNunno2009}, we define the Malliavin calculus framework for infinite horizon.

\begin{definition}[Malliavin-Sobolev spaces on infinite horizon]
\label{def:malli-spaces}
Let $p\ge 1$ and $k\ge 1$. We define $\DD^{k,p}_\lambda$ as the
completion of the set of smooth random variables with respect to the norm
\[
\|F\|_{\DD^{k,p}_\lambda}^p = \EE_\QQ[|F|^p] + \sum_{j=1}^k \EE_\QQ\left[\int_0^\infty \cdots \int_0^\infty e^{-\lambda(t_1+\cdots+t_j)} \|D_{t_1,\ldots,t_j}F\|^p dt_1 \cdots dt_j\right],
\]
where $D$ denotes the Malliavin derivative. For processes, define $\LL^{k,p}_\lambda$ as the space of progressively measurable processes $u$ with:
\[
\begin{aligned}
\|u\|_{\LL^{k,p}_\lambda}^p
&=
\EE_\QQ\left[
\int_0^\infty e^{-\lambda t}|u_t|^p\,dt
\right]  \\
&\quad
+ \sum_{j=1}^k
\EE_\QQ\left[
\int_0^\infty\cdots\int_0^\infty
e^{-\lambda(t_1+\cdots+t_j)}
\|D_{t_1,\ldots,t_j}u_t\|^p
\,dt_1\cdots dt_j\,dt
\right].
\end{aligned}
\]
\end{definition}

\subsection{Clark-Ocone Formula under Change of Measure}

We need a version of the Clark-Ocone formula that accounts for the change of measure from $\PP$ to $\QQ$. This requires careful handling of the Malliavin derivatives of the density process.

\begin{lemma}[Malliavin derivative of the density]\label{lem:malliavin-density}
Under Assumptions \ref{ass:A1} and \ref{ass:A3}, the density process
\[
M(t) = \frac{d\QQ}{d\PP}\bigg|_{\FC_t}
\]
satisfies:
\begin{enumerate}
    \item For the Brownian part: $D_s M(t) = M(t)\xi(s) \mathbf{1}_{[0,t]}(s)$.
    \item For the jump part: 
    \[
    D_{s,\zeta} M(t) = M(t) \frac{D_{s,\zeta} \beta(s,\zeta)}{1+\beta(s,\zeta)} \mathbf{1}_{[0,t]}(s), \quad s\le t.
    \]
\end{enumerate}
Moreover, $M(t) \in \DD^{1,2}_\lambda$ for all $t \geq 0$.
\end{lemma}

\begin{theorem}[Clark-Ocone formula under change of measure on infinite horizon]\label{thm:clark-ocone-change}
Let $F \in \DD^{1,2}_\lambda$ be $\FC_\infty$-measurable with $\EE_\QQ[|F|^2] < \infty$. Then:
\begin{align*}
F &= \EE_\QQ[F] 
+ \int_0^\infty \EE_\QQ\Big[D_s F - F \int_s^\infty D_s \xi(r) dB_\QQ(r) \ \Big|\ \FC_s \Big] dB_\QQ(s) \\
&\quad + \int_0^\infty \int_{\RB} \EE_\QQ\Big[ F(\tilde{H}_s-1) + \tilde{H}_s D_{s,\zeta} F \ \Big|\ \FC_s \Big] \tilde{N}_\QQ(ds,d\zeta),
\end{align*}
where
\begin{align*}
\tilde{H}_s &= \exp\Bigg( 
\int_0^s \int_{\RB} \Big[ D_{s,\zeta} \beta(r,\zeta') 
+ \ln \Big(1 - \frac{D_{s,\zeta} \beta(r,\zeta')}{1+\beta(r,\zeta')} \Big)(1+\beta(r,\zeta')) \Big] \nu(d\zeta') dr \\
&\quad + \int_0^s \int_{\RB} \ln \Big(1 - \frac{D_{s,\zeta} \beta(r,\zeta')}{1+\beta(r,\zeta')} \Big) \tilde{N}_\QQ(dr,d\zeta') \Bigg).
\end{align*}
\end{theorem}

\subsection{Application to BSVIE Solution}

\begin{theorem}[Malliavin representation of $(Z,K)$ with random coefficients]\label{thm:ZK-random}
Let $(Y,Z,K)$ be the adapted M-solution of the infinite-horizon BSVIE~\eqref{eq:BSVIE}, and define
\[
U(t) = \int_t^\infty (\Phi(t,s)Y(s) + h(t,s)) ds - Y(t).
\]
Assume that
\begin{enumerate}
    \item $U(t) \in \DD^{1,2}_\lambda$ for all $t \geq 0$.
    \item $\xi \in \LL^{1,2}_\lambda$ and $\beta(\cdot,\zeta) \in \LL^{1,2}_\lambda$ for all $\zeta$.
    \item The integrability condition holds:
    \[
    \EE_\QQ\Bigg[ \int_0^\infty e^{-\lambda s} \Big( |D_s U(t)|^2 + |U(t)|^2 \int_s^\infty |D_s \xi(r)|^2 dr \Big) ds \Bigg] < \infty.
    \]
\end{enumerate}
Then for $0 \le t \le s < \infty$ and $\zeta \in \RB$:
\begin{align}
Z(t,s) &= \EE_\QQ\Big[ D_s U(t) - U(t) \int_s^\infty D_s \xi(r) dB_\QQ(r) \ \Big|\ \FC_s \Big], \\
K(t,s,\zeta) &= \EE_\QQ\Big[ U(t)(\tilde{H}_s-1) + \tilde{H}_s D_{s,\zeta} U(t) \ \Big|\ \FC_s \Big].
\end{align}
\end{theorem}

\end{document}